\documentclass[12pt]{amsart}
\usepackage[utf8]{inputenc}
\usepackage{tikz} 
\usepackage{tikz-cd} 
\usepackage{amsthm} 
\usepackage{amsmath} 
\usepackage{amssymb}
\usepackage{mathtools} 
\usepackage{hyperref}
\usepackage[capitalise,noabbrev]{cleveref} 
\usepackage{geometry}
\usepackage{comment} 
\usepackage{microtype} 

\newcommand{\N}{\mathsf{N}}
\newcommand{\T}{\mathsf{T}}
\newcommand{\E}{\mathsf{E}}

\newcommand{\Q}{\mathsf{Q}}
\newcommand{\LN}{\mathsf{LN}}
\newcommand{\LA}{\mathsf{LA}}
\newcommand{\LX}{\mathsf{LX}}
\newcommand{\A}{\mathsf{A}}
\newcommand{\B}{\mathsf{B}}
\newcommand{\G}{\mathsf{G}}
\newcommand{\LT}{\mathsf{LT}}
\newcommand{\LB}{\mathsf{LB}}
\newcommand{\LQ}{\mathsf{LQ}}
\newcommand{\X}{\mathsf{X}}

\newcommand{\M}{\mathsf{M}}
\renewcommand{\L}{\mathsf{L}}
\renewcommand{\P}{\mathsf{P}}
\newcommand{\LE}{\mathsf{LE}}

\renewcommand{\ker}{\mathsf{ker}}
\newcommand{\PAN}{\mathsf{P_AN}}
\newcommand{\PAQ}{\mathsf{P_AQ}}
\newcommand{\PAT}{\mathsf{P_AT}}
\newcommand{\K}{\mathsf{K}}
\newcommand{\R}{\mathsf{R}}

\newtheorem{theorem}{Theorem}[section]
\newtheorem{lemma}[theorem]{Lemma}
\newtheorem{prop}[theorem]{Proposition}
\newtheorem{corollary}[theorem]{Corollary}

\theoremstyle{definition}
\newtheorem{definition}[theorem]{Definition}

\newtheorem{example}[theorem]{Example}
\newtheorem{remark}[theorem]{Remark}

\def\red{\color{black}}
\def\blue{\color{black}}
\def\rouge{\color{red}}
\usepackage[
    style=nature,
    intitle=true,
    citestyle=alphabetic,
    backend=biber
  ]{biblatex}

\defbibenvironment{bibliography}
  {\list
     {\printtext[labelalphawidth]{%
        \printfield{labelprefix}%
        \printfield{labelalpha}%
        \printfield{extraalpha}}}
     {\setlength{\labelwidth}{1.8cm}
      \setlength{\labelsep}{0cm}
      }
      }
  {\endlist}
  {\item}

\title{Non-existence of fiberwise localization for crossed modules}

\author{Olivia Monjon}
\author{J\'er\^ome Scherer}
\address{Mathematics, Ecole Polytechnique F\'ed\'erale de Lausanne, EPFL, Switzerland}
\email{olivia.monjon@epfl.ch}
\address{Mathematics, Ecole Polytechnique F\'ed\'erale de Lausanne, EPFL, Switzerland}
\email{jerome.scherer@epfl.ch}
\author{Florence Sterck}
\address{Institut de Recherche en Mathématique et Physique, Universit\'e catholique de Louvain, Belgium and Département de Mathématique, Université Libre de Bruxelles, Belgium}
\email{ florence.sterck@uclouvain.be}

\subjclass[2010]{18G45, 55P60, 55R65, 18E13. }
\keywords{Crossed modules, Localization functors, Fiberwise localizations, Acyclicity.} 

\begin{document}

\begin{abstract} We prove that localization functors of crossed modules of groups do not always admit
fiberwise (or relative) versions. To do so we characterize the existence of a fiberwise localization
by a certain normality condition and compute explicit examples and counter-examples. In fact, some
nullification functors do not behave well and we also prove that the fiber
of certain nullification functors, known as acyclization functors in other settings such
as groups or spaces, is not acyclic.
\end{abstract}

\maketitle

\section*{Introduction} 
As soon as localization functors have been introduced in homotopy theory and algebra, fiberwise techniques have been developed in order to reduce certain
questions about extensions to easier ones. In homotopy theory for example May, \cite{MR554323}, highlighted the fundamental role of fiberwise localization in Sullivan's influential article
\cite{MR442930} on the Adams conjecture. Let us also mention the use of fiberwise plus constructions in algebraic $K$-theory, which explains how Quillen's
plus construction is related to the lower $K$-theory groups, see Berrick's book \cite{MR649409}, or Arlettaz' survey \cite[Section~3]{MR1775748}.
Fiberwise localization also plays a prominent part in Farjoun's \cite[Chapter~I]{MR1392221}, and the conjunction between general fiberwise techniques and
$K$-theoretical motivations led then Berrick and Farjoun to their work \cite{MR1997044}. 
Finally, let us mention that in the modern approach to homotopy theory by $\infty$-categorical methods
fiberwise localization appears in the recent work of Gepner and Kock \cite{MR3641669}, in the form of factorization \emph{systems} in relation to the univalence axiom.

In group theory, Casacuberta and Descheemaeker noticed in \cite{Casacuberta} that localization functors 
admit a relative version, where
group extensions replace fibration sequences. One way to obtain such a construction is to adapt Hilton's construction from \cite{MR732481}. Original
applications were related to algebraic $K$-theory again, via the plus-construction, while more recent computations include purely group theoretical work 
by Flores and the second author,  \cite{MR3881250}, and the study of conditional flatness, see \cite{FS}. In the latter, fiberwise localization was a 
key tool to understand the difference between homotopical localization for spaces and group theoretical localization.

However, many arguments one can perform for groups make sense in any semi-abelian category in the sense of Janelidze, M\'arki, and Tholen, 
\cite{JMT}. They provided axioms that capture the properties not only of the categories of groups, but also non-unital rings, crossed modules, Lie algebras, 
cocommutative Hopf algebras over a field \cite{GSV}, etc. Roughly speaking, semi-abelian categories are to groups what abelian categories are to abelian 
groups. In joint work with Gran, \cite{GranScherer}, the
second author studied thus the behavior of localization functors in an arbitrary semi-abelian category with respect to extensions, with a specific focus on the
preservation of certain properties under pullbacks. Most statements depend on the existence of fiberwise localization functors and since they always
exist in the category of groups, this project grew out of the desire to understand what happens in the next most obvious category of interest to both algebraists and
homotopy theorists, namely crossed modules of groups.

Crossed modules have been studied by Whitehead in \cite{whitehead}. They serve as a combinatorial model for connected $2$-types,
i.e. spaces with vanishing homotopy groups in degrees $\geq 3$ and have been used extensively in homotopy theory, see for example
work of Brown and Higgins, \cite{BH}. Crossed modules also enjoy nice categorical properties and have been studied from the algebraic
viewpoint. Our starting point is Norrie's \cite{Norrie}, where she establishes most of the constructions we need for crossed modules
of groups. Janelidze extended this widely and defined the notion of internal crossed modules in any semi-abelian category \cite{Jan}. 
He generalized the result of Brown and Spencer \cite{BS} and proved that the category of internal crossed modules in a semi-abelian category 
$\mathcal{C}$ is equivalent to the category of internal groupoids in $\mathcal{C}$, which forms again a semi-abelian category \cite{BG}.

When working with extensions of crossed modules and localization functors, one wishes to have a fiberwise or a relative
version at hand. It means that if $\sf 1 \rightarrow N \rightarrow T \rightarrow Q \rightarrow 1$
is a short exact sequence of crossed modules, we are looking for a natural transformation to a new sequence
$\sf 1 \rightarrow LN \rightarrow \E \rightarrow Q \rightarrow 1$ where the morphism $\sf N \rightarrow LN$ is the
localization coaugmentation $\ell^\N$ and $\sf T \rightarrow \E$ is inverted by $\sf L$. In the case of groups \cite{Casacuberta}, Casacuberta and 
Descheemaeker gave an explicit description of such a construction for any short exact sequence of groups by using the notions of actions and 
semi-direct products. With similar tools introduced by Norrie \cite{Norrie}, we thought it would be possible to adapt the construction for groups to 
the case of crossed modules. Surprisingly, this approach does not work even when we restrict our setting to the case of localization functors $\sf L$
for which  the coaugmentation $\sf T \rightarrow LT$ is a regular epimorphism for all crossed modules~$\sf T$. It turns out that fiberwise localization
unexpectedly fails to exist in general.

\medskip

\noindent
{\bf Theorem~\ref{thmfiberwiselocalization}.}
\emph{
Let $\sf L \colon \sf XMod \to XMod$ be a regular-epi localization functor. Let us consider the following exact sequence of crossed modules.
\begin{equation*}
\begin{tikzpicture}[descr/.style={fill=white},baseline=(A.base),scale=0.8] 
\node (A) at (0,0) {$\T$};
\node (B) at (2.5,0) {$\Q$};
\node (C) at (-2.5,0) {$\N$};
\node (O1) at (-4.5,0) {$1$};
\node (O2) at (4.5,0) {$1$};
\path[-stealth,font=\scriptsize]
(C.east) edge node[above] {$\kappa$} (A.west)
;
\path[-stealth,font=\scriptsize]
(O1) edge node[above] {$ $} (C)
(B) edge node[above] {$ $} (O2)
(A) edge node[above] {$\alpha$} (B);
\end{tikzpicture} 
 \end{equation*}
 This exact sequence admits a fiberwise localization if and only if $\kappa(\ker(\ell^{\N}))$ is a normal subcrossed module of~$\T$.
}

\medskip

This helps us to understand that even harmless looking nullification functors such as $\sf P_{\sf X \mathbb Z}$, the functor
that kills all copies of the crossed module $0 \rightarrow \mathbb Z$ concentrated in one degree, do not satisfy this
normality condition, as we prove in Theorem~\ref{nonexistencelocfibr}. To our knowledge this is the first example of this
kind. We show finally in Proposition~\ref{nullificationnonacyclic} that the fiber
of the same nullification functor, known as acyclization functors in other settings such
as groups or spaces, is not acyclic in general. The two phenomena were known to be related (fiberwise nullification
and acyclization), see \cite{GranScherer}, but again, this is the first concrete example we know of. 

\medskip

\noindent
 {\bf Acknowledgements.} We would like to thank Marino Gran for sharing his insight on the phenomena studied here,
Mariam Pirashvili for fruitful conversations in Strasbourg, and the referee for their thorough reading and helpful suggestions. The third author’s research
is supported by a FRIA doctoral grant no. 27485 of the Communauté française
de Belgique. This work started when the third author visited the first and second authors at the \'Ecole Polytechnique F\'ed\'erale de Lausanne. This visit was possible thanks to a Gustave Boël–Sofina Fellowship.

\section{The semi-abelian category of crossed modules}
In this first section, we provide the basic definitions, notation, and constructions 
we use in the category of crossed modules. We describe in particular pushouts and cokernels
for crossed modules. We follow Norrie \cite{Norrie} and Brown-Higgins \cite{BH}.

\begin{definition}\label{xmod grp}\cite{whitehead}
A \emph{crossed module of groups} is given by a morphism of groups $\partial^\X \colon X_1 \to X_2$ endowed with an action by automorphisms of groups of $X_2$ on $X_1$, 
denoted by $X_2 \times X_1 \to X_1 \colon (b,x) \mapsto \;^{b}x$, such that for any $b$ in $X_2$ and any $x$, $y$ in $X_1$,
\begin{equation}
\partial^\X(\;^b x) = b\partial^\X(x)b^{-1},
\end{equation}
\begin{equation}\label{peiffer}
^{\partial^\X(x)}y = xyx^{-1}.
\end{equation}
\end{definition}

Hence a crossed module is a triple $(X_1,X_2,\partial^\X)$ and we will sometimes refer to $\partial^X$
as the \emph{connecting morphism}. For the sake of readability, we will also use the notation $\X$ for such a crossed module. We give several examples of this notion.

\begin{example}\label{examplexmod}
\begin{enumerate}
    \item The inclusion of a normal subgroup $N$ of $M$ is a crossed module where the action of $M$ on $N$ is given by the conjugation $^m n = mnm^{-1}$. As a particular example, the identity morphism $G = G$ 
provides a way to construct a crossed module from a single
group.


\item For any group $G$, the inclusion of the trivial group 
$1 \rightarrowtail G$ endowed with the trivial action is a 
crossed module.


\item Let $g : A \rightarrow B$  be a surjective morphism of groups such that its kernel is included in the center of $A$, i.e.\ 
$ker(g) \subseteq Z(A)= \{ x \in A \mid xa = ax, \; \forall a \in A\}$.
There exists an action of $B$ on $A$ via $ B \times A \to A \colon (b,a) \mapsto xax^{-1}$, where $g(x) =b$. This action is well defined since $ker(g) \subseteq Z(A)$.
One can check that this gives to the morphism $g$ the structure of a crossed module.
\end{enumerate}
\end{example}

\begin{definition} \label{morphxmod}
Let $\N := (N_1,N_2,\partial^{\N})$ and $\M := (M_1,M_2,\partial^{\M})$ be two crossed modules, a \emph{morphism of crossed modules} is given by a pair of group homomorphisms $\alpha := (\alpha_1, \alpha_2) : (N_1 \rightarrow M_1 , N_2 \rightarrow M_2)$ such that the two following diagrams commute
\[
\begin{tikzpicture}[descr/.style={fill=white},scale=0.8,baseline=(A.base)]
\node (A) at (0,0) {$N_2$};
\node (B) at (0,2) {$N_1$};
\node (C) at (2,2) {$M_1$};
\node (D) at (2,0) {$M_2$};
  \path[-stealth,font=\scriptsize]
 (B.south) edge node[left] {$\partial^{\N}$} (A.north) 
 (C.south) edge node[right] {$\partial^{\M}$} (D.north) 
  (B.east) edge node[above] {$\alpha_1$} (C.west) 
 (A.east) edge node[below] {$\alpha_2$} (D.west);
\end{tikzpicture}\hspace{1cm} \begin{tikzpicture}[descr/.style={fill=white},yscale=0.8,baseline=(A.base)]
\node (A) at (0,0) {$M_2 \times M_1$};
\node (B) at (0,2) {$N_2 \times N_1$};
\node (C) at (2,2) {$N_1$};
\node (D) at (2,0) {$M_1.$};
  \path[-stealth,font=\scriptsize]
 (B.south) edge node[left] {$(\alpha_2, \alpha_1)$} (A.north) 
 (C.south) edge node[right] {$\alpha_1$} (D.north) 
  (B.east) edge node[above] {$ $} (C.west) 
 (A.east) edge node[below] {$ $} (D.west);
\end{tikzpicture}
\]
where the horizontal arrows in the right diagram are the respective group actions of the two crossed modules.
\end{definition}

The two definitions above give rise to the category
$\mathsf{XMod}$ of crossed modules of groups. 
We remark that there is an embedding of the category of groups in this category via two functors {\blue which have been introduced on the objects in \cref{examplexmod}. These functors} are respectively left and right adjoint to the truncation functor $Tr \colon \sf XMod \to Grp$ that sends a crossed module $\T := (T_1,T_2,\partial^\T)$
to $T_2$. This will help us to import group theoretical results into $\sf XMod$. 

\begin{lemma}\label{lemma_lefttruncation}
The functor $\sf X\colon Grp \to XMod$ which sends a group $G$ to the crossed module ${\sf X} G = ( 1,G, 1)$ reduced to the group $G$ at level~$2$ 
is left adjoint to the truncation functor $Tr$. The functor ${ \sf R \colon  Grp \to XMod } \colon G \mapsto (G,G,Id_G) $ is right adjoint to the truncation functor $ Tr$.
\end{lemma}

\begin{proof}
We have natural isomorphisms ${ Hom_{\sf XMod}}( {\sf X}G, \T) \cong { Hom_{{\sf Grp }}}(G, T_2)$ for any crossed module~$\sf T$
and likewise ${ Hom_{\sf XMod}}( {\sf T}, {\sf R}G) \cong { Hom_{{\sf Grp }}}(T_2, G)$.
\end{proof}

There is an obvious notion of subcrossed module, see \cite{Norrie}. One simply requires the subobject to be made levelwise of subgroups, the connecting
homomorphism and the action are induced by the given connecting homomorphism and action. In other words, we have the following definition.

\begin{definition}\label{defsubxmod}
Let $(i_1,i_2) \colon \N \to \T$ be a morphism of crossed modules, if $i_1$ and $i_2$ are group inclusions we say that $\N$ is a \emph{subcrossed module} of $\T$. 
\end{definition}

For $\N$ a subcrossed module of $\T = (T_1,T_2,\partial^\T)$,
we introduce a subgroup of $T_1$ denoted as follows \[ [N_2,T_1] := \langle \;^{n_2}t_1t_1^{-1} \mid  t_1  \in T_1, n_2 \in N_2 \rangle . \] 
Let us notice that the commutator subgroup of $T_1$, $[T_1,T_1] = \langle t_1t_1't_1^{-1}t_1'^{-1} \mid t_1, t_1' \in T_1\rangle$, is included in 
$[T_2,T_1]$. {\blue Indeed, we have the following inclusions via condition \eqref{peiffer} of crossed modules: \[ [T_1,T_1] = [\partial(T_1),T_1] \subseteq [T_2,T_1].\] }
\newline

What is less obvious maybe and the source of interesting phenomena in $\sf XMod$ that one cannot see within the category of groups
is the notion of normality and thus of quotient or cokernel.
  

\begin{definition}\label{defnormalcrossed}
A subcrossed module $\N := (N_1,N_2,\partial^\N)$ of $\T := (T_1,T_2,\partial^\T)$ is a \emph{normal subcrossed module} if the following conditions hold
\begin{enumerate}
    \item $N_2$ is a normal subgroup of $T_2$;
    \item For any $t_2 \in T_2$ and $n_1 \in N_1$; $^{t_2}n_1 \in N_1$.
    \item $[N_2,T_1] \subseteq N_1$
\end{enumerate}
\end{definition}

We recall from \cite{BH,LG} the construction of the quotient. It illustrates the fact that colimits are not straightforward to construct in the category 
of crossed modules, but in the case of cokernels we have an explicit formula that will be very useful in concrete computations.
Let $G \times H \to H$  be an action of groups and $S$ a subgroup of $H$, we denote by $S_G$ the closure of $S$ via the action of $G$:
\[ 
S_G := \langle \;^gs \mid g \in G, s \in S \rangle . 
\]

\begin{definition}\label{def epi reg}
Let $f\colon \sf H \rightarrow \sf T$ be a morphism of crossed modules. The \emph{cokernel} of $f$ is the crossed module
${\sf coker} f$ given by the following morphism of crossed modules
\[
\begin{tikzpicture}[descr/.style={fill=white},yscale=0.8]
\node (A) at (0,0) {$T_2$};
\node (B) at (0,2) {$T_1$};
\node (C) at (4,2) {${T_1}/\big({f_1(H_1)_{T_2}[f_2(H_2)_{T_2},T_1]}\big)$};
\node (D) at (4,0) {$coker(f_2)$};
  \path[-stealth,font=\scriptsize]
 (B.south) edge node[left] {$\partial^{\T}$} (A.north) 
 (C.south) edge node[right] {$\tilde{\partial}^\T$} (D.north) 
  (B.east) edge node[above] {$ $} (C.west) 
 (A.east) edge node[below] {$ $} (D.west);
\end{tikzpicture}\]
where $\tilde{\partial}^\T$ is induced by the universal property of the cokernel of groups.
\end{definition}

\begin{remark}
Note that when ${\sf H}$ is a normal subcrossed module of $\sf T$ the above definition of cokernel 
coincides simply with the levelwise quotient by the normal subgroups $H_1 \triangleleft T_1$ and $H_2 \triangleleft T_2$.
\end{remark}

The kernel of a morphism of crossed modules is defined ``component-wise" as in the category of groups. More precisely, let $(\alpha_1, \alpha_2) \colon \A \to \B $ be a morphism of crossed modules. The kernel of this morphism denoted by $\ker(\alpha_1,\alpha_2)$ is given by \[ (ker(\alpha_1), ker(\alpha_2), \hat{\partial}^\A),\]
where $\hat{\partial}^\A$ is induced by the universal property of $ker(\alpha_2)$.

\begin{remark}
\label{rem:mononormal}
We recall that in $\sf XMod$ it is equivalent to being a normal subcrossed module or the kernel of some morphism (a normal monomorphism) \cite{Norrie}.
\end{remark}

More generally, all limits are computed ``component-wise'' as in the category of groups.
For example, pullbacks in $\sf XMod$ are built as follows \cite{LG}. Let $f \colon \T \to \Q$ and $g \colon \Q' \to \Q$ be two morphisms of crossed modules. Then the pullback of $f$ along $g$ is given by the following square
\[
\begin{tikzpicture}[descr/.style={fill=white},yscale=0.8]
\node (A) at (0,0) {$\T$};
\node (B) at (0,2) {$\T'$};
\node (C) at (2,2) {$\Q'$};
\node (D) at (2,0) {$\Q$};
  \path[-stealth,font=\scriptsize]
 (B.south) edge node[left] {$\pi_\T$} (A.north) 
 (C.south) edge node[right] {$g$} (D.north) 
  (B.east) edge node[above] {$\pi_{\Q'}$} (C.west) 
 (A.east) edge node[below] {$f$} (D.west);
\end{tikzpicture}
\]
The object part $\T'$ of the pullback is built component-wise as in the case of groups 
\[ 
(T_1\times_{Q_1}Q'_1, T_2\times_{Q_2}Q'_2, \partial'),
\]
where $\partial'$ and the action are induced by the universal property of the pullbacks in $\sf Grp$. The projections are the natural ones, given also component-wise.\newline

In contrast to the limits, which are built component-wise, colimits are not. In particular, the construction of cokernels is not as straightforward as the case of groups, as we saw 
in \cref{def epi reg}. We refer to \cite[Proposition~11]{BH} for the description of pushouts in $\sf XMod$. Note that thanks to the adjunction of \cref{lemma_lefttruncation}, the ``second level'' of the pushout of crossed modules is always constructed as in $\sf Grp$.

The category of crossed modules is semi-abelian, as shown in \cite{JMT}. This notion has been introduced by Janelidze, M\'arki, and Tholen in \cite{JMT}.

 Semi-abelian categories enjoy many nice properties we will use in the following 
 sections, such as the traditional homological lemmas, \cite{BB}, the Split Short Five Lemma, \cite{Bourn}, the Noether Isomorphism Theorems, \cite{BB}, 
 and that one can recognize pullbacks by looking at kernels or cokernels, \cite[Lemmas 4.2.4 and 4.2.5]{BB}. For the sake of completeness, we recall Lemma 4.2.4 in \cite{BB}, which will be useful several times in this article.
 
 \begin{prop}[Lemma 4.2.4 \cite{BB}]\label{usefulpropopointed}
Let $\mathcal{C}$ be a pointed category. We consider the following diagram where $\kappa$ is the kernel of $\alpha$
\[
\begin{tikzpicture}[descr/.style={fill=white},baseline=(A.base),scale=0.8] 
\node (A) at (0,0) {$T'$};
\node (B) at (2.5,0) {$Q'$};
\node (C) at (-2.5,0) {$N'$};
\node (A') at (0,-2) {$T$};
\node (B') at (2.5,-2) {$Q$};
\node (C') at (-2.5,-2) {$N$};
\node (X) at (-1.25,-1) {$(1)$};
\node (Y) at (1.25,-1) {$(2)$};
\path[-stealth,font=\scriptsize]
(B.south) edge node[right] {$ w$}  (B'.north)
 (C.south) edge node[left] {$ u $}  (C'.north)
(A.south) edge node[left] {$ v$} (A'.north)
(C'.east) edge node[above] {$\kappa$} (A'.west)
(A') edge node[above] {$\alpha$} (B') ([yshift=2pt]A'.east)
(C.east) edge node[above] {$\kappa'$} (A.west)
(A) edge node[above] {$\alpha'$} (B);
\end{tikzpicture} 
\]
\begin{enumerate}
    \item If $w$ is a monomorphism then $\kappa' =ker(\alpha')$ if and only if (1) is a pullback;
    \item When $(2)$ is a pullback and $\alpha' \circ \kappa'$ is the zero morphism, $\kappa'$ is the kernel of $\alpha'$ if and only if $u$ is an isomorphism.
\end{enumerate}

\end{prop}

\begin{remark}
\label{rem:eprireg}
The relevant categorical notion of epimorphism in this context is that 
of regular epimorphism (a coequalizer of a pair of parallel arrows).
In the category of crossed modules, a morphism $f=(f_1,f_2)$ is a regular epimorphism if and only if it is surjective on each component, i.e. $f_1$ and $f_2$ are surjective group homomorphisms \cite[Proposition~2.2]{LLR}. Moreover, we note that each surjective morphism is an epimorphism but there exist epimorphisms that are not surjective . Since $\sf XMod$ is a pointed protomodular category, regular epimorphisms and normal epimorphisms (the cokernel of some morphism) coincide.
\end{remark}
\section{Localization functors}
We recall the definition of localization functors, which we describe for crossed modules. 

\begin{definition}\label{localization}
A \emph{localization} functor in the category of crossed modules is a coaugmented
idempotent functor $\sf L \colon XMod \to XMod$. The coaugmentation is a natural
transformation $\ell \colon {\sf Id} \rightarrow \L$. Both $\ell^{\LX}$ and $ \L \ell^\X$ are
isomorphisms and in particular we have $\ell^{\LX} = \L \ell^\X$, see \cite[Proposition~1.1]{MR1796125}.
\end{definition}

\begin{definition}\label{local}
Let $\sf L$ be a localization functor. A crossed module $\sf T$ is $\sf L$-\emph{local}
if the coaugmentation morphism $\ell^{\sf T}\colon \sf T \rightarrow \sf L \sf T$ is an
isomorphism. A morphism $f\colon \sf N \rightarrow \sf M$ is an $\sf L$-\emph{equivalence} 
if ${\sf L} f$ is an isomorphism.
\end{definition}

Here are a few basic and useful properties of $\L$-equivalences.

\begin{lemma}\label{lemma_colimitandequivalences}
\begin{enumerate}
    \item The pushout of an $\sf L$-equivalence is an $\L$-equivalence.
     \item The composition of $\L$-equivalences is an $\L$-equivalence.
    \item A $\kappa$-filtered colimit of a diagram $\sf T_\beta$ of $\sf L$-equivalences $\sf T_\beta \rightarrow  T_{\beta +1}$ for all successor ordinals 
    $\beta +1 < \kappa$ yields an $\sf L$-equivalence $\sf T_0 \rightarrow T_\kappa = \sf colim_{\beta < \kappa} T_\beta$.
    \item Let ${\sf F}$ be an $I$-indexed diagram of $\sf L$-equivalences in the category of morphisms of crossed modules. 
    Then the colimit $\sf colim_I F$ is an $\L$-equivalence.
\end{enumerate}
\end{lemma}

\begin{proof}
Property (4) is \cite[Proposition~1.3]{MR1796125}. Properties (1), (2) and (3) follow.
\end{proof}

Sometimes one comes across a localization functor by finding a full
reflexive subcategory $\mathcal L$ of $\sf XMod$ (for example of all abelian crossed
modules). The pair of adjoint functors $\sf U\colon \mathcal L \leftrightarrows \sf XMod\colon \sf F$
yields a localization functor $\sf L = \sf FU$. This is the approach taken by
Cassidy, H\'ebert, and Kelly in \cite{MR779198}.
There are other situations where one constructs localization functors by fixing
a set of morphisms which are required to become isomorphisms after localization.
In this way any set of morphisms $S$ defines a localization functor ${\sf L}_S$
inverting the elements of $S$.
By letting $f$ be the coproduct of all morphisms in $S$ it is good enough
to study localization functors
of the form ${\sf L}_f$. The existence and construction of such
functors follow from very general results, see for example
Bousfield's foundational work \cite{MR478159} or the more
recent account by Hirschhorn \cite{Hirschhorn} in a model categorical setting.

\begin{definition}\label{f-local}
Let $f$ be a morphism of crossed modules. A crossed module
$\sf T$ is \emph{${\sf L}_f$-local} if $Hom(f, \T)$ is an isomorphism.
A morphism $g$ in $\sf XMod$ is an \emph{${\sf L}_f$-equivalence} if
$Hom(g, \T)$ is an isomorphism for any ${\sf L}_f$-local $\sf T$.
\end{definition}

In this context, local objects and local equivalences coincide with
the notion introduced previously in Definition~\ref{local}. The properties
we stated in Proposition~\ref{lemma_colimitandequivalences} are then
analogous to \cite[Proposition~1.2.21]{Hirschhorn} for the first claim, 
by the universal property of a pushout, and the second statement follows 
by induction as \cite[Proposition~1.2.20]{Hirschhorn}. Note that $\L$-equivalences and $\L$-local objects defined in \cref{local} can also be expressed via their universal properties as in \cref{f-local}. 

\begin{prop}\label{universalprop_localization}
A crossed module
$\sf T$ is \emph{${\sf L}$-local} if $Hom(h, \T)$ is an isomorphism for any $\L$-equivalence~$h$.
A morphism $g$ in $\sf XMod$ is an \emph{${\sf L}$-equivalence} if
$Hom(g, \T)$ is an isomorphism for any ${\sf L}$-local $\sf T$.
\end{prop}

\begin{definition}\label{null}
Let $f$ be a morphism of the form $\sf A \to 1$, where $\sf A$ is a crossed
module. The localization functor ${\sf L}_f$ is written $\sf P_A$ and is called a
\emph{nullification} functor. One calls the local objects $\sf A$-\emph{null}
or $\sf A$-\emph{local} and a crossed module $\sf T$ such that  $\sf P_A T = 1$
is called $\sf A$-\emph{acyclic}.
\end{definition}

In this article localization functors having the property
that the coaugmentation  morphism is always a regular epimorphism will
play a central role.

\begin{definition}\label{defregepi}
A localization functor is called a \emph{regular-epi localization} if for any crossed module $\T$ the morphism $\ell^\T \colon \T \to \L\T$ is a regular epimorphism. 
\end{definition} 


To prove that any nullification functor is a regular-epi localization we will need to describe in a precise way how $\sf P_A T$ is constructed by
successively killing all maps from $\sf A$ to~$\sf T$.

\begin{prop}\label{propnullify}
Let $\sf A$ be a crossed module. Then $\sf P_A$ is a 
regular-epi localization.
\end{prop}

\begin{proof}
Let $\sf T$ be a crossed module. 
Consider the coproduct $\coprod \sf A$
taken over $Hom(\sf A, \sf T)$ and form an evaluation morphism
$ev\colon \coprod \sf A \rightarrow \sf T$ where the 
component map $\sf A \to T$ indexed by the morphism $g$ is
precisely~$g$. We then define $\sf T_0 = T$ and $\sf T_1$ is
the pushout of the diagram
\[
\sf 1 \leftarrow \coprod A \xrightarrow{ev} T_0
\]
In other words, $\sf T_1$ is the cokernel of the evaluation map
as described in \cref{def epi reg}, so $\sf T_0 \rightarrow \sf T_1$
is a regular epimorphism. It is also a $\sf P_A$-equivalence, being the
pushout of the $\sf P_A$-equivalence $\coprod \sf A \rightarrow \sf 1$ (see \cref{lemma_colimitandequivalences}).

For each successor ordinal $\beta+1$, the crossed module $\sf T_{\beta+1}$ 
is defined as above, $\sf T_{\beta+1} = (\sf T_{\beta})_1$, and for 
a limit ordinal $\beta$, we define 
$\sf T_{\beta} = colim_{\gamma < \beta} T_{\gamma}$.
As a composition of regular epimorphisms is again a regular epimorphism
we obtain by induction that the morphism $\sf T \rightarrow \sf T_{\beta+1}$
is a regular epimorphism. The case of limit ordinals is taken care of
by the fact that regular epimorphisms are coequalizers, which are preserved
under colimits.

To finish the proof we have to explain why this process stops, which
follows from Quillen's small object argument,
\cite[Proposition~10.5.16]{Hirschhorn} or \cite{MR0223432}, as is well-known.

We choose $\lambda$ to be the first infinite ordinal greater
than the number of chosen generators of $A_1$ and $A_2$. This
implies that a morphism out of $\sf A$ is determined by strictly
less than $\lambda$ images of elements so that $\sf A$ is
$\lambda$-small with respect to 
$\sf T_\lambda = colim_{\beta < \lambda} T_\beta$,
i.e. any morphism $g\colon \sf A \rightarrow T_\lambda$ factors through
some intermediate stage $\sf T_\beta$ for a certain $\beta < \lambda$. 
This comes from the fact that filtered colimits are created in sets 
and every chosen generator $t$ must be sent to some $\sf T_{\alpha_t}$.
The ordinal $\beta$ is then the union of all $\alpha_t$'s.

Therefore $g$ becomes trivial in $\sf T_{\beta +1}$, which shows
that the crossed module $\sf P_A T := \T_\lambda$ is $\sf A$-null.
The map $\sf T \rightarrow P_A T$ is an $\sf A$-equivalence by
Lemma~\ref{lemma_colimitandequivalences} since
it is obtained by iterating pushouts along $\sf P_A$-equivalences.
\end{proof}

Sometimes it is handy to rely on our group theoretical knowledge
to construct simple examples of localization functors and how they
behave on crossed modules. Recall the functor $\sf X$ from
Lemma~\ref{lemma_lefttruncation}.

\begin{prop}\label{prop_Lfromgroups}
Let $\varphi\colon G \to H$ be a group homomorphism.
The localization functor ${\sf L}_{{\sf X} \varphi}$ verifies
${\sf L}_{{\sf X} \varphi} {\sf X} T \cong {\sf X L}_\varphi T$ for any group $T$.
\end{prop}

\begin{proof}
The adjunction in Lemma~\ref{lemma_lefttruncation} tells us that
a crossed module $\sf A$ is ${\sf X}\varphi$-local if and only
if $A_2$ is a $\varphi$-local group. In particular, ${\sf X L}_\varphi T$ is ${\sf X}\varphi$-local. Moreover, $\sf X$ sends
${\sf L}_\varphi$-equivalences to ${\sf L}_{{\sf X} \varphi}$-equivalences.
\end{proof}

In this situation, it is thus easy to localize a crossed module that is reduced to a group (granted that we know how to localize groups), but the effect on arbitrary crossed modules can be more surprising because the groups at level one are linked to the groups at level two via the connecting homomorphism.

The end of this section is devoted to illustrating the notion of localization functor by
a handful of natural examples. We give a non-exhaustive list of examples of localization functors of crossed modules: some of them are obtained by 
using the construction of nullification functors (\cref{example_PZ}, \cref{example_PZ2}), some are built using well-known constructions (\cref{Ab}, 
\cref{Nil}), and some are induced by an adjunction (\cref{I}). It is interesting to notice that some of the following examples already appear in the literature. 
In particular, the subcategories induced by the local objects of \cref{example_PZ} and \cref{example_PZ2} form a hereditary torsion theory \cite{BGtorsion}. 
We start with an important functor as it will be the key player in our counter-examples.

\begin{example}\label{example_PZ}
The nullification functor $\P_{\sf X\mathbb{Z}}$ with respect to
the crossed module $\sf X\mathbb{Z}$ is described as follows 
\[ 
\P_{\sf X\mathbb{Z}} \left( 
\begin{tikzpicture}[descr/.style={fill=white},yscale=0.7,baseline=(O.base)] 
\node (O) at (-5,-1) {$ $};
\node (F) at (-4.5,0) {$N_1$};
\node (F') at (-4.5,-2) {$N_2$};
\path[-stealth,font=\scriptsize]
(F.south) edge node[left] {$ \partial $} (F'.north);
\end{tikzpicture} 
 \right) = 
\begin{tikzpicture}[descr/.style={fill=white},yscale=0.7,baseline=(O.base)] 
\node (O) at (-1,-1) {$ $};
\node (C) at (-2.5,0) {$ N_1 / [N_2, N_1]$};
\node (C') at (-2.5,-2) {$1$};
\path[-stealth,font=\scriptsize]
(C.south) edge node[left] {$  $} (C'.north);
\end{tikzpicture} 
 \]
In this example, the construction detailed in Proposition~\ref{propnullify}
can be done in a single step, since the coproduct of all morphisms ${\sf X}\mathbb Z \to \N$ comes from a surjective group homomorphism 
$F \to N_2$, where $F$ is a free group. Hence, in the first step of the construction we construct the quotient of ${\sf X} F \to \N$ by killing its 
normal closure as introduced in Definition~\ref{def epi reg}. This kills obviously $N_2$ and quotients $N_1$ out by 
$inc(1)_{N_2}[Id(N_2)_{N_2},N_1] = [N_2, N_1]$. 

The map from $\sf N$ to this quotient is an 
${\sf X} \mathbb Z$-equivalence being the pushout of an
equivalence and this quotient is local (the bottom group is trivial).
From the point of view of reflexive subcategories, this localization functor
corresponds to the reflector associated with the subcategory of crossed modules
of the form $(A, 1)$ where $A$ is any abelian group and the connecting homomorphism
is the trivial homomorphism.
\end{example}

\begin{example}\label{example_PZ2}
We give the explicit description of the functor of nullification respectively to the crossed module $\mathbb{Z} \to 0$. The functor is then defined as follows:
  \[ 
\P_{ \mathbb{Z} \to 0} \left( 
\begin{tikzpicture}[descr/.style={fill=white},yscale=0.7,baseline=(O.base)] 
\node (O) at (-5,-1) {$ $};
\node (F) at (-4.5,0) {$N_1$};
\node (F') at (-4.5,-2) {$N_2$};
\path[-stealth,font=\scriptsize]
(F.south) edge node[left] {$ \partial^\N $} (F'.north);
\end{tikzpicture} 
 \right) = 
\begin{tikzpicture}[descr/.style={fill=white},yscale=0.7,baseline=(O.base)] 
\node (O) at (-1,-1) {$ $};
\node (C) at (-2.5,0) {$\partial^\N(N_1) $};
\node (C') at (-2.5,-2) {$N_2$};
\path[-stealth,font=\scriptsize]
(C.south) edge node[descr] {$ inc $} (C'.north);
\end{tikzpicture} 
  \]
Let us notice that the objects of the subcategory of $P_{ \mathbb{Z} \to 0}$-local objects are inclusions of normal subgroups as in \cref{examplexmod}. In terms of internal groupoids, this construction can be seen as the reflector of the category of internal groupoids to the category of internal equivalence relations (see example 2.5 in \cite{GE}).
 \end{example}

The next two examples already appeared in \cite{Norriethesis}. 

\begin{example}\label{Ab}
Another localization functor is given by the abelianization functor. We denote this functor by $ \sf Ab \colon \sf  XMod \to XMod$ and it is defined as follows
\[ 
{\sf Ab} \left( 
\begin{tikzpicture}[descr/.style={fill=white},yscale=0.7,baseline=(O.base)] 
\node (O) at (-5,-1) {$ $};
\node (F) at (-4.5,0) {$N_1$};
\node (F') at (-4.5,-2) {$N_2$};
\path[-stealth,font=\scriptsize]
(F.south) edge node[left] {$ \partial $} (F'.north);
\end{tikzpicture} 
 \right) = 
\begin{tikzpicture}[descr/.style={fill=white},yscale=0.7,baseline=(O.base)] 
\node (O) at (-1,-1) {$ $};
\node (C) at (-2.5,0) {$N_1/ [N_2,N_1] $};
\node (C') at (-2.5,-2) {$N_2/ [N_2,N_2]$};
\path[-stealth,font=\scriptsize]
(C.south) edge node[left] {$  \tilde{\partial} $} (C'.north);
\end{tikzpicture} 
  \]
  
  \end{example}

 \begin{example}\label{Nil}
The abelianization functor can be generalized and we can define the nilpotent functors. Indeed, in \cite{Norriethesis}, Norrie defined the notion of lower central series. 
Hence, we can quotient any crossed module by the $k$-th term in its lower central series and obtain a functor 
  \[ 
  { \sf Nil_k} \colon \sf XMod \to \sf XMod \colon \G \mapsto \G / \Gamma_k(\G) 
  \]
  where $\sf \Gamma_k(G)$ is the $k$-th term in the lower central series of $\G$.
  We give an explicit description of the functor ${ \sf Nil_2} \colon \sf XMod \to XMod$.
   \[ 
{ \sf Nil_2} \left( 
\begin{tikzpicture}[descr/.style={fill=white},yscale=0.7,baseline=(O.base)] 
\node (O) at (-5,-1) {$ $};
\node (F) at (-4.5,0) {$N_1$};
\node (F') at (-4.5,-2) {$N_2$};
\path[-stealth,font=\scriptsize]
(F.south) edge node[left] {$ \partial $} (F'.north);
\end{tikzpicture} 
 \right) = 
\begin{tikzpicture}[descr/.style={fill=white},yscale=0.7,baseline=(O.base)] 
\node (O) at (-1,-1) {$ $};
\node (C) at (-2.5,0) {${N_1}/{<[[N_2,N_2],N_1],[[N_2,[N_2,N_1]]>} $};
\node (C') at (-2.5,-2) {${N_2}/{[[N_2,N_2],N_2]}$};
\path[-stealth,font=\scriptsize]
(C.south) edge node[left] {$  \tilde{\partial} $} (C'.north);
\end{tikzpicture} 
   \]
%

\end{example} 

\begin{example}\label{C}
We can also consider the localization functor ${ \sf C} \colon \sf XMod \to XMod$ defined by
 \[ 
{ \sf C } \left( 
\begin{tikzpicture}[descr/.style={fill=white},yscale=0.7,baseline=(O.base)] 
\node (O) at (-5,-1) {$ $};
\node (F) at (-4.5,0) {$N_1$};
\node (F') at (-4.5,-2) {$N_2$};
\path[-stealth,font=\scriptsize]
(F.south) edge node[left] {$ \partial^\N $} (F'.north);
\end{tikzpicture} 
 \right) = 
\begin{tikzpicture}[descr/.style={fill=white},yscale=0.7,baseline=(O.base)] 
\node (O) at (-1,-1) {$ $};
\node (C) at (-2.5,0) {$1 $};
\node (C') at (-2.5,-2) {$N_2/\partial^\N(N_1)$};
\path[-stealth,font=\scriptsize]
(C.south) edge node[left] {$   $} (C'.north);
\end{tikzpicture} 
  \]
In fact, the localization functor ${ \sf C}$ is exactly the nullification functor $\P_{ \mathbb{Z} \xrightarrow{id} \mathbb{Z}}$.
\end{example}

\begin{example}\label{I}
We give a final example of a functor of localization of crossed modules ${ \sf I} \colon \sf XMod \to XMod$
 \[ 
{ \sf I} \left( 
\begin{tikzpicture}[descr/.style={fill=white},yscale=0.7,baseline=(O.base)] 
\node (O) at (-5,-1) {$ $};
\node (F) at (-4.5,0) {$N_1$};
\node (F') at (-4.5,-2) {$N_2$};
\path[-stealth,font=\scriptsize]
(F.south) edge node[left] {$ \partial^\N $} (F'.north);
\end{tikzpicture} 
 \right) = 
\begin{tikzpicture}[descr/.style={fill=white},yscale=0.7,baseline=(O.base)] 
\node (O) at (-1,-1) {$ $};
\node (C) at (-2.5,0) {$N_2$};
\node (C') at (-2.5,-2) {$N_2$};
\path[-stealth,font=\scriptsize]
(C.south) edge node[left] {$  Id $} (C'.north);
\end{tikzpicture} 
  \]
This localization functor is induced by the adjunction between the truncation functor, introduced in \cref{lemma_lefttruncation}, $Tr \colon {\sf XMod \to Grp} \colon (T_1,T_2,\partial^\T) \mapsto T_2$ and the functor ${\sf R \colon Grp \to XMod }$ sending $T$ to $(T,T,Id)$, which plays here the role of right adjoint of $Tr$.
 \end{example}

\begin{remark}
The functors considered in Examples \ref{example_PZ}, \ref{example_PZ2}, \ref{Ab}, \ref{Nil} and \ref{C} are regular-epi localizations. In particular, every nullification functor is a regular-epi localization. However, the converse is not true as illustrated by 
the functor $\sf Ab$ of \cref{Ab}. There is a large collection of regular-epi localization functors. Indeed, similarly to the observation in \cite{CasacubertaAnderson}, if $f$ is a regular epimorphism, then the functor $L_f$ is a regular-epi localization functor.


\end{remark}

From now on, every localization functor that we consider is a regular-epi localization.

\section{Construction of fiberwise localization}
In this section, we study the concept of fiberwise localization for an exact sequence of crossed modules. 
We show how to construct such a fiberwise localization when the exact sequence satisfies a certain normality condition,
reminiscent of condition $(N)$ in \cite{GE}. The authors introduced this condition $(N)$ to obtain a weaker context (than abelian categories), in which a 
torsion theory gives rise to a monotone-light factorization system. At the end of the section, we investigate this condition in detail and show it is necessary to obtain a fiberwise localization.

\begin{definition}
Let ${ \sf L} \colon {\sf XMod} \to {\sf XMod}$ be a localization functor. An exact sequence 
\[ 
\begin{tikzpicture}\label{exactseq}[descr/.style={fill=white},scale=0.8,baseline=(A.base),
xscale=1.2] 
\node (A) at (0,0) {$\T$};
\node (B) at (2.5,0) {$\Q$};
\node (C) at (-2.5,0) {$\N$};
\node (O1) at (-4.5,0) {$1$};
\node (O2) at (4.5,0) {$1$};
\path[-stealth,font=\scriptsize]
(C.east) edge node[above] {$\kappa$} (A.west)
;
\path[-stealth,font=\scriptsize]
(O1) edge node[above] {$ $} (C)
(B) edge node[above] {$ $} (O2)
(A) edge node[above] {$\alpha$} (B);
\end{tikzpicture} 
\]
admits a fiberwise localization if there exists such a commutative diagram of exact sequences
\begin{center}
\begin{tikzpicture}[descr/.style={fill=white},baseline=(A.base),scale=0.8] 
\node (A) at (0,0) {$\T$};
\node (B) at (2.5,0) {$\Q$};
\node (C) at (-2.5,0) {$\N$};
\node (A') at (0,-2) {$\E$};
\node (B') at (2.5,-2) {$\Q$};
\node (C') at (-2.5,-2) {$\LN$};
\node (O1) at (-4.5,0) {$1$};
\node (O1') at (-4.5,-2) {$1$};
\node (O2) at (4.5,0) {$1$};
\node (O2') at (4.5,-2) {$1$};
\path[-stealth,font=\scriptsize]
(C.east) edge node[above] {$\kappa$} (A.west)
;
\draw[commutative diagrams/.cd, ,font=\scriptsize]
(B) edge[commutative diagrams/equal] (B');
\path[-stealth,font=\scriptsize]
(C'.east) edge node[above] {$j $} (A'.west)
(O1) edge node[above] {$ $} (C)
(B) edge node[above] {$ $} (O2)
(B') edge node[above] {$ $} (O2')
(O1') edge node[above] {$ $} (C')
 (C.south) edge node[left] {$ \ell^{\N} $}  (C'.north)
(A') edge node[above] {$ p$} (B') ([yshift=2pt]A'.east)
(A) edge node[above] {$\alpha$} (B);
\path[-stealth,font=\scriptsize]
(A.south) edge node[left] {$ g $} (A'.north);
\end{tikzpicture} 
 \end{center} where $g$ is an $\L$-equivalence.
\end{definition}


%
We give a sufficient condition on this exact sequence to construct a fiberwise localization.

\begin{prop}
Let $\L \colon \sf XMod \to XMod$ be a regular-epi localization functor. Any exact sequence of crossed modules such that $\kappa(\ker(\ell^\N))$ is a 
normal subcrossed module of $\T$,
\begin{center}
\begin{tikzpicture}[descr/.style={fill=white},scale=0.8,baseline=(A.base),
xscale=1.2] 
\node (A) at (0,0) {$\T$};
\node (B) at (2.5,0) {$\Q$};
\node (C) at (-2.5,0) {$\N$};
\node (O1) at (-4.5,0) {$1$};
\node (O2) at (4.5,0) {$1$};
\path[-stealth,font=\scriptsize]
(C.east) edge node[above] {$\kappa$} (A.west)
;
\path[-stealth,font=\scriptsize]
(O1) edge node[above] {$ $} (C)
(B) edge node[above] {$ $} (O2)
(A) edge node[above] {$\alpha$} (B);
\end{tikzpicture} 
 \end{center}
admits a fiberwise localization.
\end{prop}

\begin{proof}
By assuming that $\ell^{\N} \colon \N \to \LN$ is a regular epimorphism and that $\kappa(\ker(\ell^\N))$ is a normal crossed module of $\T$, 
we can construct the following diagram

\begin{center}
\begin{tikzpicture}[descr/.style={fill=white},scale=0.8,baseline=(A.base)] 
\node (C'') at (-2.5,2) {$\ker(\ell^{\N})$};
\node (A) at (0,0) {$\T$};
\node (B) at (2.5,0) {$\Q$};
\node (C) at (-2.5,0) {$\N$};
\node (A') at (0,-2) {$\T/ \kappa(\ker(\ell^{\N}))$};
\node (B') at (2.5,-2) {$\Q$};
\node (C') at (-2.5,-2) {$\LN$};
\node (O1) at (-4.5,0) {$1$};
\node (O2) at (4.5,0) {$1$};
\draw[commutative diagrams/.cd, ,font=\scriptsize]
(B) edge[commutative diagrams/equal] (B');
\path[-stealth,font=\scriptsize]
  (C'') edge node[descr] {$  \kappa' $}  (A)
 (C'') edge node[left] {$  $}  (C.north)
(C.east) edge node[above] {$\kappa$} (A.west)
;
\path[-stealth,font=\scriptsize]
(C'.east) edge node[above] {$ $} (A'.west)
(O1) edge node[above] {$ $} (C)
(B) edge node[above] {$ $} (O2)
 (C.south) edge node[left] {$ \ell^{\N} $}  (C'.north)
(A') edge node[above] {$ $} (B') 
(A) edge node[above] {$\alpha$} (B);
\path[-stealth,font=\scriptsize,dashed,font=\scriptsize]
(A.south) edge node[left] {$ f $} (A'.north);
\end{tikzpicture} 
 \end{center}
 where $f$ is the cokernel of the normal morphism $\kappa'=\kappa\mid_{\ker(\ell^\N)} \colon \ker(\ell^{\N}) \to \T$. The lower sequence is exact via the first and second isomorphism theorems for crossed modules (Theorems 2.1 and 2.2 in \cite{Norriethesis}). 
\newline

 To end this proof, we need to show that $f \colon \T \to \T / \kappa(\ker(\ell^{\N}))$ is an $\L$-equivalence. Let $\G$ be a local object in $\sf XMod$ and $\beta \colon \T \to \G$ be a morphism of crossed modules. We need to prove that there exists a unique morphism of crossed modules $\tilde{\beta}$ from $\T/ \kappa(\ker(\ell^{\N}))$ to $\G$ (\cref{universalprop_localization}). This morphism is induced by the universal property of the cokernel $f$ and the universal property of the localization. To use the universal property of the cokernel
 $f$ \begin{center}
\begin{tikzpicture}[descr/.style={fill=white},baseline=(A.base),
scale=0.8] 
\node (A) at (0,0) {$\T$};
\node (B) at (2.5,0) {${\T}/{\kappa(\ker(\ell^{\N}))}$};
\node (C) at (-2.5,0) {$\ker(\ell^{\N})$};
\node (A') at (2.5,-2) {$\G$};
\path[-stealth,font=\scriptsize]
(C.east) edge node[above] {$\kappa'$} (A.west)
;
\path[-stealth,font=\scriptsize]
(A) edge node[above] {$f$} (B)
(A) edge node[descr] {$ \beta $} (A');
\path[dashed,-stealth,font=\scriptsize]
(B) edge node[right] {$ \tilde{\beta} $} (A');
\end{tikzpicture} 
 \end{center}
 we need to prove that $\beta \circ \kappa' $ is the zero morphism.
  This can be deduced from the commutativity of the following diagram
   
   \begin{center}
\begin{tikzpicture}[descr/.style={fill=white},baseline=(A.base),scale=0.8] 
\node (C'') at (-2.5,2) {$\ker(\ell^{\N})$};
\node (A) at (0,0) {$\T$};
\node (C) at (-2.5,0) {$\N$};
\node (A') at (0,-2) {$\G$};
\node (C') at (-2.5,-2) {$\LN$};
\path[-stealth,font=\scriptsize]
  (C'') edge node[descr] {$ \kappa' $}  (A)
 (C'') edge node[left] {$  $}  (C.north)
(C.east) edge node[above] {$\kappa$} (A.west)
;
\path[-stealth,font=\scriptsize]
 (C.south) edge node[left] {$ \ell^{\N} $}  (C'.north)
(A.south) edge node[left] {$ \beta $} (A'.north);
\path[-stealth,font=\scriptsize,dashed,font=\scriptsize]
(C'.east) edge node[above] {$\psi $} (A'.west);
\end{tikzpicture} 
 \end{center}
 where $\psi$ is induced by the universal property of the localization.
 So we can conclude that $f$ is an $\L$-equivalence. Hence we built a fiberwise localization.  
\end{proof}

In the previous proposition, we gave a condition on the exact sequence of crossed modules ensuring the existence of a fiberwise localization. 
Now, we prove that this condition is actually mandatory.

\begin{prop}
Let $\L \colon \sf XMod \to XMod$ be a regular-epi localization functor. If the following exact sequence of crossed modules
\begin{equation}\label{exactseqfiber}
\begin{tikzpicture}[descr/.style={fill=white},baseline=(A.base),scale=0.8] 
\node (A) at (0,0) {$\T$};
\node (B) at (2.5,0) {$\Q$};
\node (C) at (-2.5,0) {$\N$};
\node (O1) at (-4.5,0) {$1$};
\node (O2) at (4.5,0) {$1$};
\path[-stealth,font=\scriptsize]
(C.east) edge node[above] {$\kappa$} (A.west)
;
\path[-stealth,font=\scriptsize]
(O1) edge node[above] {$ $} (C)
(B) edge node[above] {$ $} (O2)
(A) edge node[above] {$\alpha$} (B);
\end{tikzpicture} 
 \end{equation}
admits a fiberwise localization, then the kernel $\kappa(\ker(\ell^{\N}))$ is a normal subcrossed module of $\T$.
\end{prop}

\begin{proof}
Suppose we have a fiberwise localization for the exact sequence \eqref{exactseqfiber} with $\ell^{\N}$ a regular epimorphism. 
It means that there exists $\E \in \sf XMod$ and a diagram

\begin{center}
\begin{tikzpicture}[descr/.style={fill=white},baseline=(A.base),scale=0.8] 
\node (A'') at (0,2) {$\ker(f)$};
\node (C'') at (-2.5,2) {$\ker(\ell^{\N})$};
\node (A) at (0,0) {$\T$};
\node (X) at (-1.25,-1) {$(1)$};
\node (B) at (2.5,0) {$\Q$};
\node (C) at (-2.5,0) {$\N$};
\node (A') at (0,-2) {$\E$};
\node (B') at (2.5,-2) {$\Q$};
\node (C') at (-2.5,-2) {$\LN$};
\node (O1) at (-4.5,0) {$1$};
\node (O1') at (-4.5,-2) {$1$};
\node (O2) at (4.5,0) {$1$};
\node (O2') at (4.5,-2) {$1$};
\path[-stealth,font=\scriptsize]

 (C'') edge node[left] {$  $}  (C.north)
(C.east) edge node[above] {$\kappa$} (A.west)
;
\draw[commutative diagrams/.cd, ,font=\scriptsize]
(B) edge[commutative diagrams/equal] (B');
\path[-stealth,font=\scriptsize]
  (C'') edge node[left] {$  $}  (A'')
  (A'') edge node[left] {$  $}  (A)
(C'.east) edge node[above] {$j $} (A'.west)
(O1) edge node[above] {$ $} (C)
(B) edge node[above] {$ $} (O2)
(B') edge node[above] {$ $} (O2')
(O1') edge node[above] {$ $} (C')
 (C.south) edge node[left] {$ \ell^{\N} $}  (C'.north)
(A') edge node[above] {$ p$} (B') ([yshift=2pt]A'.east)
(A) edge node[above] {$\alpha$} (B);
\path[-stealth,font=\scriptsize]
(A.south) edge node[left] {$ f $} (A'.north);
\end{tikzpicture} 
 \end{center} We use \cref{usefulpropopointed} to observe that (1) is a pullback and that $\ker(\ell^{\N})$ is isomorphic to $\ker(f)$. 
Hence, we can conclude that $\kappa(\ker(\ell^{\N}))$ is a normal subcrossed module of $\T$.
\end{proof} 

Thanks to the two previous propositions we can now state the following theorem.

\begin{theorem}\label{thmfiberwiselocalization}
Let $\sf L \colon \sf XMod \to XMod$ be a regular-epi localization functor. An exact sequence of crossed modules
\begin{equation}
\begin{tikzpicture}[descr/.style={fill=white},baseline=(A.base),scale=0.8] 
\node (A) at (0,0) {$\T$};
\node (B) at (2.5,0) {$\Q$};
\node (C) at (-2.5,0) {$\N$};
\node (O1) at (-4.5,0) {$1$};
\node (O2) at (4.5,0) {$1$};
\path[-stealth,font=\scriptsize]
(C.east) edge node[above] {$\kappa$} (A.west)
;
\path[-stealth,font=\scriptsize]
(O1) edge node[above] {$ $} (C)
(B) edge node[above] {$ $} (O2)
(A) edge node[above] {$\alpha$} (B);
\end{tikzpicture} 
 \end{equation}
admits a fiberwise localization if and only if $\kappa(\ker(\ell^{\N}))$ is a normal subcrossed module of~$\T$. \hfill{\qed}
\end{theorem}

\begin{remark}
This theorem can be generalized and holds in any semi-abelian category, \cite{GranScherer}. Our restricted setup, namely that of crossed modules allows us
to produce concrete examples where one can actually check the normality condition.

\end{remark}
We emphasize that the normality condition for $\kappa(\ker(\ell^\N))$ in \cref{thmfiberwiselocalization} is actually the same as the one called $(N)$ in \cite{GE}. It is interesting to notice that 
this condition appears in completely different contexts and for different purposes. In our case, it is the adequate condition to obtain a fiberwise localization. 
In \cite{GE}, they require that a torsion theory (in a normal category) satisfies this condition to obtain a monotone-light factorization system.
We investigate this condition of normality in the category of crossed modules and reexpress it as an easier statement to verify. 
 Indeed, we prove that several items in the definition of normal subcrossed modules always hold.

\begin{prop}
 Let $\kappa\colon N := (N_1,N_2, \partial^{\N}) \rightarrow \T :=(T_1,T_2,\partial^{\T})$ be a normal monomorphism of crossed modules, then $\kappa(\ker(\ell^\N))$ is a subcrossed module of $\T$ and we have the two following properties:
\begin{itemize}
    \item[$(i)$] $\kappa_2(ker(\ell_2^{\N}))$ is a normal subgroup of $T_2$.
    \item[$(ii)$] For any $t_2 \in T_2$ and $n_1 \in ker(\ell_1^{\N})$ then $^{t_2}\kappa_1(n_1) \in \kappa_1(ker(\ell_1^{\N}))$.
\end{itemize}
\end{prop}

\begin{proof}

Using \cref{rem:mononormal} we identify $\N$ with the normal subcrossed module $\kappa(\N)$ and thus omit the use of $\kappa$ in this proof.
It is straightforward to see that  $\ker(\ell^\N)$ is a subcrossed module of $\T$ via \cref{defsubxmod}.
To show properties $(i)$ and $(ii)$ of the lemma we will use the following construction. Since $\N$ is a normal subcrossed module of $\T$ 
we have an induced action of $\T$ on $\N$ as explained in Definition 1.3.5 in \cite{LG}. In terms of crossed modules, it implies that we have 
a ``conjugation'' morphism of crossed modules $c_{t_2} := (\theta_{t_2},\sigma_{t_2}) \in Aut(N_1,N_2,\partial^{\N})$ depending on an element 
$t_2 \in T_2$ defined by

\begin{align*}
    \theta_{t_2}: \ &N_1 \rightarrow N_1 \hspace{2cm} \sigma_{t_2}: N_2 \rightarrow N_2 \\
     & n_1 \mapsto \ ^{t_2}n_1 \hspace{2.3cm} n_2 \mapsto t_2n_2t_2^{-1}
\end{align*}
If we consider the morphism $ c_{t_2} \colon \N \to \N$ in $\sf XMod$ we can construct the following diagram
 \[ \begin{tikzpicture}[descr/.style={fill=white},baseline=(A.base),scale=0.8]
 \node (X) at (-2.5,0) {$\ker(\ell^\N)$};
\node (Y) at (-2.5,2) {$\ker(\ell^\N)$};
\node (A) at (0,0) {$\N$};
\node (B) at (0,2) {$\N$};
\node (C) at (2.5,2) {$\LN$};
\node (D) at (2.5,0) {$\LN$};
  \path[-stealth,font=\scriptsize,dashed]
  (Y) edge node[left] {$c_{t_2}|_{\ker(\ell^\N)} $} (X); 
  \path[>-stealth,font=\scriptsize]
  (Y) edge node[below] {$ $} (B)
  (X) edge node[below] {$ $} (A); 
  \path[-stealth,font=\scriptsize]
  (A.east) edge node[below] {$\ell^{\N}$} (D.west) 
 (C.south) edge node[right] {${\sf L} c_{t_2}$} (D.north) 
 (B.south) edge node[descr] {$c_{t_2}$} (A.north) 
 (B.east)  edge node[above] {$\ell^{\N}$} (C.west);
\end{tikzpicture}
\] 
By definition of the kernel and its universal property, $c_{t_2}$ restricts to the kernel, which implies that  
properties $(i)$ and $(ii)$ hold. 
\end{proof}

This implies that the normality condition of $\kappa(\ker(\ell^\N))$ in the crossed module $\T$ can be expressed as follows.
\begin{corollary}
\label{Corollaryconditionnormalite}
 Let $\kappa\colon N := (N_1,N_2, \partial^{\N}) \rightarrow \T :=(T_1,T_2,\partial^{\T})$ be a normal monomorphism of crossed modules, then $\kappa(\ker(\ell^\N))$ is a normal subcrossed module of $\T$ if and only if we have the following inclusion 
\begin{equation}\label{Conditionnormalite}
[\kappa_2(ker(\ell_2^\N)),T_1] \subseteq \kappa_1(ker(\ell_1^\N))
\end{equation}
\end{corollary}


\section{Examples and counter-examples for fiberwise localizations}
In this section, we illustrate the construction of fiberwise localization in $\sf XMod$ by using the normality condition. It is first interesting to notice that for some particular localization functors, condition \eqref{Conditionnormalite} is always satisfied. For instance, it is the case if the functor $\L \colon \sf XMod \to XMod$ preserves monomorphisms. Surprisingly, there exist examples for which the normality condition does not hold. In contrast to the case of groups or topological spaces, in $\sf XMod$ fiberwise localization does not always exist.

\begin{lemma}\label{preservemono}
Let $\L \colon \sf XMod \to XMod$ be a regular-epi localization functor that preserves monomorphisms and \[\begin{tikzpicture}[descr/.style={fill=white},baseline=(A.base),scale=0.8] 
\node (A) at (0,0) {$\T$};
\node (B) at (2.5,0) {$\Q$};
\node (C) at (-2.5,0) {$\N$};
\node (O1) at (-4.5,0) {$1$};
\node (O2) at (4.5,0) {$1$};
\path[-stealth,font=\scriptsize]
(C.east) edge node[above] {$\kappa$} (A.west)
;
\path[-stealth,font=\scriptsize]
(O1) edge node[above] {$ $} (C)
(B) edge node[above] {$ $} (O2)
(A) edge node[above] {$\alpha$} (B);
\end{tikzpicture} \] be an exact sequence of crossed modules. 
Then $\kappa(\ker(\ell^{\N}))$ is a normal subcrossed module of $\T$. 
\end{lemma}

\begin{proof}
%
%
%
Consider the following diagram where ${\sf L} \kappa$ is a monomorphism.
\[ \begin{tikzpicture}[descr/.style={fill=white},baseline=(A.base),yscale=0.8]
\node (O1) at (-2,2) {$ 1$};
\node (O2) at (-2,0) {$ 1$};
\node (A) at (0,0) {$\ker(\ell^\T)$};
\node (X) at (1,1) {$(1)$};
\node (B) at (0,2) {$\ker(\ell^\N)$};
\node (C) at (2,2) {$\N$};
\node (D) at (2,0) {$\T$};
\node (E) at (4,2) {$\LN$};
\node (F) at (4,0) {$\LT$};
  \path[>-stealth,font=\scriptsize]
  (A.east) edge node[below] {$ $} (D.west) 
 (C.south) edge node[right] {$ \kappa $} (D.north) 
 (B.south) edge node[left] {$ $} (A.north) 
 (B.east)  edge node[above] {$ $} (C.west)
 (E.south) edge node[right] {$ \L \kappa$} (F.north);
   \path[-stealth,font=\scriptsize]
   (O1) edge node[below] {$ $} (B)
   (O2) edge node[below] {$ $} (A)
    (C.east)  edge node[above] {$ \ell^\N$} (E.west)
 (D.east) edge node[below] {$\ell^\T$} (F.west);
\end{tikzpicture}
\]
Since $\L \kappa$ is a monomorphism then (1) is a pullback by \cref{usefulpropopointed}. It implies that $\kappa(\ker(\ell^\N))$ is a normal subcrossed 
module of $\T$ as it can be seen as the intersection of the normal subcrossed modules $\N$ and $\ker(\ell^\T)$ of $\T$.
\end{proof}

We now give some examples of localization functors and exact sequences for which we can apply the fiberwise localization construction.

\begin{example}
 Let us consider the nullification functor with respect to the crossed module $\mathbb{Z} \to 0$ as defined in \cref{example_PZ2}.
 \[ 
{ \sf P_{ \mathbb{Z} \to 0}} \left( 
\begin{tikzpicture}[descr/.style={fill=white},yscale=0.7,baseline=(O.base)] 
\node (O) at (-5,-1) {$ $};
\node (F) at (-4.5,0) {$N_1$};
\node (F') at (-4.5,-2) {$N_2$};
\path[-stealth,font=\scriptsize]
(F.south) edge node[left] {$ \partial^\N $} (F'.north);
\end{tikzpicture} 
 \right) = 
\begin{tikzpicture}[descr/.style={fill=white},yscale=0.7,baseline=(O.base)] 
\node (O) at (-1,-1) {$ $};
\node (C) at (-2.5,0) {$\partial^\N(N_1)$};
\node (C') at (-2.5,-2) {$N_2$};
\path[-stealth,font=\scriptsize]
(C.south) edge node[descr] {$ inc $} (C'.north);
\end{tikzpicture} 
  \]
Monomorphisms can be described ``component-wise'' as two monomorphisms in the category of groups. We can see that if $\N $ is a 
subcrossed module of $\M$, then $ \partial^{\N}(N_1)$ is included in $ \partial^{\M}(M_1)$ since $\partial^{\N}$ is the restriction of $\partial^{\M}$. 
The other conditions are trivial so we can conclude that ${ \sf P_{ \mathbb{Z} \to 0}}$ preserves monomorphisms. This observation was also made in Example 2.5 in \cite{GE}, where they noticed that it implies that the condition that they called $(N)$ holds. In our case, condition $(N)$ is the necessary and sufficient condition to obtain fiberwise localization (see \cref{thmfiberwiselocalization}). It implies that for any exact sequence of crossed modules there exists a fiberwise localization for the functor ${ \sf P_{ \mathbb{Z} \to 0}}$.
  \end{example}
 The functor ${ \sf C} \colon \sf XMod \to XMod$ introduced in \cref{C} is also an example of a localization functor that satisfies always the condition of \cref{thmfiberwiselocalization}. 
 \begin{example}
Let \[ \begin{tikzpicture}[descr/.style={fill=white},baseline=(A.base),scale=0.8] 
\node (A) at (0,0) {$T_1$};
\node (B) at (2.5,0) {$Q_1$};
\node (C) at (-2.5,0) {$N_1$};
\node (A') at (0,-2) {$T_2$};
\node (B') at (2.5,-2) {$Q_2$};
\node (C') at (-2.5,-2) {$N_2$};
\node (O1) at (-4.5,0) {$1$};
\node (O1') at (-4.5,-2) {$1$};
\node (O2) at (4.5,0) {$1$};
\node (O2') at (4.5,-2) {$1$};
\path[-stealth,font=\scriptsize]
(O1) edge node[above] {$ $} (C)
(B) edge node[above] {$ $} (O2)
(B') edge node[above] {$ $} (O2')
(O1') edge node[above] {$ $} (C')
(B.south) edge node[right] {$ \partial^{\Q}$}  (B'.north)
 (C.south) edge node[left] {$ \partial^{\N} $}  (C'.north)
(A.south) edge node[left] {$ \partial^{\T}$} (A'.north)
(C'.east) edge node[above] {$\kappa_2$} (A'.west)
(A') edge node[above] {$\alpha_2$} (B') ([yshift=2pt]A'.east)
(C.east) edge node[above] {$\kappa_1$} (A.west)
(A) edge node[above] {$\alpha_1$} (B);
\end{tikzpicture} 
\] 
be an exact sequence of crossed modules. We consider the functor ${ \sf C} \colon \sf XMod \to XMod$ and prove that $\kappa(\ker(\ell^\N))$ is always normal in $\T$. Indeed, we just need to verify \eqref{Conditionnormalite} for  $\kappa(\ker(\ell^\N)) = (N_1,\partial^\N(N_1)),\tilde{\partial}^\N)$. 
 We have the following inclusions 
\[ 
[\partial^\N(N_1), T_1] \subseteq [N_2,T_1] \subseteq N_1
\]
since $\N$ normal in $\T$ and we conclude that for the localization functor ${ \sf C}$ there always exists a fiberwise localization. It is interesting to notice that even if this functor admits a fiberwise localization it does not necessarily preserve monomorphisms. We illustrate this statement via the following example where $S_4$ is the symmetric group of order 4 and $A_4$ the alternating group.
\[ 
{\sf C} \left( 
\begin{tikzpicture}[descr/.style={fill=white},yscale=0.7,baseline=(O.base)]
\node (O) at (-5,-1) {$ $};
\node (F) at (-4.5,0) {$A_4$};
\node (F') at (-4.5,-2) {$S_4$};
\node (E) at (-2.5,0) {$S_4$};
\node (E') at (-2.5,-2) {$S_4$};
\path[-stealth,font=\scriptsize]
(F) edge node[left] {$   $} (E)
(F) edge node[left] {$   $} (F');
\draw[commutative diagrams/.cd, ,font=\scriptsize]
(F') edge[commutative diagrams/equal] (E')
(E) edge[commutative diagrams/equal] (E');
\end{tikzpicture} 
 \right) = 
\begin{tikzpicture}[descr/.style={fill=white},yscale=0.7,baseline=(O.base)] 
\node (O) at (-5,-1) {$ $};
\node (F) at (-4.5,0) {$1$};
\node (F') at (-4.5,-2) {$\mathbb{Z}/2\mathbb{Z}$};
\node (E) at (-2.5,0) {$1$};
\node (E') at (-2.5,-2) {$1$};
\path[-stealth,font=\scriptsize]
(F) edge node[left] {$   $} (E)
(F') edge node[left] {$   $} (E')
(F) edge (F')
(E) edge (E');
\end{tikzpicture} 
\]

 \end{example}

  \begin{example}
   We consider the following two crossed modules $\R A_4$ and $\R S_4$ where $\R$ is the functor defined in \cref{lemma_lefttruncation} . We can verify that $\R A_4$ is a normal subcrossed module of $\R S_4$. So we have the following exact sequence.

\begin{equation}\label{exactA4}
\begin{tikzpicture}[descr/.style={fill=white},baseline=(A.base),scale=0.8] 
\node (A) at (0,0) {$S_4$};
\node (B) at (2.5,0) {$\mathbb{Z}/2\mathbb{Z}$};
\node (C) at (-2.5,0) {$A_4$};
\node (A') at (0,-2) {$S_4$};
\node (B') at (2.5,-2) {$\mathbb{Z}/2\mathbb{Z}$};
\node (C') at (-2.5,-2) {$A_4$};
\node (O1) at (-4.5,0) {$1$};
\node (O1') at (-4.5,-2) {$1$};
\node (O2) at (4.5,0) {$1$};
\node (O2') at (4.5,-2) {$1$};
\draw[commutative diagrams/.cd, ,font=\scriptsize]
(A) edge[commutative diagrams/equal] (A')
(B) edge[commutative diagrams/equal] (B')
(C) edge[commutative diagrams/equal] (C');
\path[-stealth,font=\scriptsize]
(O1) edge node[above] {$ $} (C)
(B) edge node[above] {$ $} (O2)
(B') edge node[above] {$ $} (O2')
(O1') edge node[above] {$ $} (C')
(C'.east) edge node[above] {$\kappa$} (A'.west)
(A') edge node[above] {$\pi$} (B') ([yshift=2pt]A'.east)
(C.east) edge node[above] {$\kappa$} (A.west)
(A) edge node[above] {$\pi$} (B);
\end{tikzpicture} 
 \end{equation}
 
 By considering the abelianization functor, as defined in \cref{Ab}, we wonder if it is possible to construct for this exact sequence a fiberwise localization. 
 First, we give an explicit description of the functor $\sf Ab$ applied to $\R A_4$. 
   \[ 
{\sf Ab} \left( 
\begin{tikzpicture}[descr/.style={fill=white},yscale=0.7,baseline=(O.base)] 
\node (O) at (-5,-1) {$ $};
\node (F) at (-4.5,0) {$A_4$};
\node (F') at (-4.5,-2) {$A_4$};
\draw[commutative diagrams/.cd, ,font=\scriptsize]
(F) edge[commutative diagrams/equal] (F');
\end{tikzpicture} 
 \right) = 
\begin{tikzpicture}[descr/.style={fill=white},yscale=0.7,baseline=(O.base)] 
\node (O) at (-2.5,-1) {$ $};
\node (C) at (-2.5,0) {${A_4}/{[A_4,A_4]} $};
\node (C') at (-2.5,-2) {${A_4}/{[A_4,A_4]}$};
\draw[commutative diagrams/.cd, ,font=\scriptsize]
(C) edge[commutative diagrams/equal] (C');
\end{tikzpicture} 
= 
\begin{tikzpicture}[descr/.style={fill=white},yscale=0.7,baseline=(O.base)] 
\node (O) at (-2.5,-1) {$ $};
\node (C) at (-2.5,0) {${A_4}/{V_4} $};
\node (C') at (-2.5,-2) {${A_4}/{V_4}$};
\draw[commutative diagrams/.cd, ,font=\scriptsize]
(C) edge[commutative diagrams/equal] (C');
\end{tikzpicture} 
\cong
\begin{tikzpicture}[descr/.style={fill=white},yscale=0.7,baseline=(O.base)] 
\node (O) at (-2.5,-1) {$ $};
\node (C) at (-2.5,0) {$\mathbb{Z}/3\mathbb{Z} $};
\node (C') at (-2.5,-2) {$\mathbb{Z}/3\mathbb{Z}$};
\draw[commutative diagrams/.cd, ,font=\scriptsize]
(C) edge[commutative diagrams/equal] (C');
\end{tikzpicture} 
  \]
  where $V_4$ denotes the Klein four-group. To be able to apply the construction of the fiberwise localization, we need that $\kappa(\ker(\ell^{\R A_4}))$ is a normal subcrossed module of $\R S_4$. As proved in \cref{Corollaryconditionnormalite}, we only need to verify condition \eqref{Conditionnormalite}, i.e.\ $[V_4,S_4]$ is included in $V_4$. Since $V_4$ is a normal subgroup of $S_4$, the equality $[V_4,S_4] \subseteq V_4 $ holds and $\R V_4$ is a normal subcrossed module of $\R S_4$.
   Hence, we have the following construction of the fiberwise localization
 \begin{center}
\begin{tikzpicture}[descr/.style={fill=white},baseline=(A.base),
scale=0.8] 
\node (C'') at (-2.5,2) {$\R V_4$};
\node (A) at (0,0) {$\R S_4$};
\node (B) at (2.5,0) {$\R (\mathbb{Z}/2\mathbb{Z})$};
\node (C) at (-2.5,0) {$\R A_4$};
\node (A') at (0,-2) {$\R S_3$};
\node (B') at (2.5,-2) {$\R (\mathbb{Z}/2\mathbb{Z})$};
\node (C') at (-2.5,-2) {$\R (\mathbb{Z}/3\mathbb{Z})$};
\node (O1) at (-4.5,0) {$1$};
\node (O1') at (-4.5,-2) {$1$};
\node (O2) at (4.5,0) {$1$};
\node (O2') at (4.5,-2) {$1$};
\path[-stealth,font=\scriptsize]

 (C'') edge node[left] {$  $}  (C.north)
(C.east) edge node[above] {$ $} (A.west)
 (C'') edge node[left] {$  $}  (A)
;
\draw[commutative diagrams/.cd, ,font=\scriptsize]
(B) edge[commutative diagrams/equal] (B');
\path[-stealth,font=\scriptsize]
(C'.east) edge node[above] {$ $} (A'.west)
(O1) edge node[above] {$ $} (C)
(B) edge node[above] {$ $} (O2)
(B') edge node[above] {$ $} (O2')
(O1') edge node[above] {$ $} (C')
 (C.south) edge node[left] {$ l^{\R A_4} $}  (C'.north)
(A') edge node[above] {$ $} (B') ([yshift=2pt]A'.east)
(A) edge node[above] {$\alpha$} (B);
\path[-stealth,font=\scriptsize]
(A.south) edge node[left] {$ f $} (A'.north);
\end{tikzpicture} 
 \end{center}

  \end{example} 
  
Now, we would like to emphasize the fact that condition \eqref{Conditionnormalite} is not trivially satisfied. The counter-example we propose
is very similar to the previous example where fiberwise abelianization was shown to exist.

\begin{theorem}\label{nonexistencelocfibr}
There exist regular-epi localization functors $\L$ and exact sequences in $\sf Xmod$ for which the fiberwise localization does not exist.
\end{theorem}

\begin{proof}
Let us consider the following exact sequence
\[ \begin{tikzpicture}[descr/.style={fill=white},baseline=(A.base),scale=0.8] 
\node (A) at (0,0) {$\R S_4$};
\node (B) at (2.5,0) {$\R (\mathbb{Z}/2\mathbb{Z})$};
\node (C) at (-2.5,0) {$\R A_4$};
\node (O1) at (-4.5,0) {$1$};
\node (O2) at (4.5,0) {$1$};
\path[-stealth,font=\scriptsize]
(O1) edge node[above] {$ $} (C)
(B) edge node[above] {$ $} (O2)
(C.east) edge node[above] {$ $} (A.west)
(A) edge node[above] {$ $} (B);
\end{tikzpicture} \] and apply the functor $\P_{\sf X \mathbb{Z}}$ defined in \cref{example_PZ}, to $\R A_4$.
 
 \[ 
\P_{\X \mathbb{Z}} \left( 
\begin{tikzpicture}[descr/.style={fill=white},yscale=0.7,baseline=(O.base)] 
\node (O) at (-5,-1) {$ $};
\node (F) at (-4.5,0) {$A_4$};
\node (F') at (-4.5,-2) {$A_4$};
\draw[commutative diagrams/.cd, ,font=\scriptsize]
(F) edge[commutative diagrams/equal] (F');
\end{tikzpicture} 
 \right) = 
\begin{tikzpicture}[descr/.style={fill=white},yscale=0.7,baseline=(O.base)] 
\node (O) at (-1,-1) {$ $};
\node (C) at (-2.5,0) {${A_4}/{[A_4, A_4]}$};
\node (C') at (-2.5,-2) {$1$};
\path[-stealth,font=\scriptsize]
(C.south) edge node[left] {$  $} (C'.north);
\end{tikzpicture} 
= 
\begin{tikzpicture}[descr/.style={fill=white},yscale=0.7,baseline=(O.base)] 
\node (O) at (-1,-1) {$ $};
\node (C) at (-2.5,0) {${A_4}/{V_4}$};
\node (C') at (-2.5,-2) {$1$};
\path[-stealth,font=\scriptsize]
(C.south) edge node[left] {$  $} (C'.north);
\end{tikzpicture} 
  \]
where $V_4 =[A_4, A_4]$ is the Klein four-group. The kernel of $\ell^{\R A_4}$ is given by the crossed module $V_4 \hookrightarrow A_4$. 
To be able to construct a fiberwise localization for the exact sequence \eqref{exactA4}, we need the image of this kernel to be a normal subcrossed 
module of $\R S_4$. This condition is not satisfied since \eqref{Conditionnormalite} does not hold. Indeed, the action of $S_4$ on $A_4$ is given by 
conjugation, hence, condition \eqref{Conditionnormalite} means that the commutator $[S_4,A_4]$ has to be included in $V_4$. 
But we have the following equality $[S_4,A_4] = A_4$, which implies that $V_4 \hookrightarrow A_4$ is not a normal subcrossed module of $\R S_4$. 
Therefore, fiberwise localization does not exist for this nullification functor and this exact sequence.
\end{proof}

We have understood at this point that localization functors of crossed modules do not behave like localization functors of groups. One other major
difference is illustrated by the behavior of the kernel of a nullification functor. In the category of groups (the homotopical analog is also true for 
topological spaces), the kernels of the nullification morphisms are acyclic. In fact, $L(ker(\ell^G))$ is trivial for any group $G$
if and only if $L$ is a nullification functor, \cite[Lemma~5.1]{FS}. In the context of crossed modules, such a characterization of nullification
functors fails.

\begin{prop}\label{nullificationnonacyclic}
There are nullification functors $\sf P_A$ in $\sf XMod$ such that the kernels of their localization morphisms are not $\sf A$-acyclic in general. 
\end{prop}

\begin{proof}
Let us consider the nullification functor with respect to the crossed module $\sf X \mathbb{Z} $ as in \cref{example_PZ2}.
Let us apply this functor to the crossed module $( D_8,D_8,Id_{D_8})$,
where $D_8$ is the dihedral group of order eight. 
  \[ 
\P_{\sf X \mathbb{Z}} \left( 
\begin{tikzpicture}[descr/.style={fill=white},yscale=0.7,baseline=(O.base)] 
\node (O) at (-5,-1) {$ $};
\node (F) at (-4.5,0) {$D_8$};
\node (F') at (-4.5,-2) {$D_8$};
\draw[commutative diagrams/.cd, ,font=\scriptsize]
(F) edge[commutative diagrams/equal] (F');
\end{tikzpicture} 
 \right) = 
\begin{tikzpicture}[descr/.style={fill=white},yscale=0.7,baseline=(O.base)] 
\node (O) at (-1,-1) {$ $};
\node (C) at (-2.5,0) {$\frac{D_8}{[D_8, D_8]}$};
\node (C') at (-2.5,-2) {$1$};
\path[-stealth,font=\scriptsize]
(C.south) edge node[left] {$  $} (C'.north);
\end{tikzpicture} 
 = 
\begin{tikzpicture}[descr/.style={fill=white},yscale=0.7,baseline=(O.base)] 
\node (O) at (-1,-1) {$ $};
\node (C) at (-2.5,0) {$\frac{D_8}{C_2}$};
\node (C') at (-2.5,-2) {$1$};
\path[-stealth,font=\scriptsize]
(C.south) edge node[left] {$  $} (C'.north);
\end{tikzpicture} 
  \]
  where $[D_8, D_8] = C_2$, which is actually the center of $D_8$. 
  Hence, the kernel of this construction is given by the crossed module $(C_2,D_8,inc)$. By applying the functor $\P_{\sf X \mathbb{Z}}$ on this crossed module we obtain the following crossed module
    \[ 
\P_{\sf X \mathbb{Z}} \left( 
\begin{tikzpicture}[descr/.style={fill=white},yscale=0.7,baseline=(O.base)] 
\node (O) at (-5,-1) {$ $};
\node (F) at (-4.5,0) {$C_2$};
\node (F') at (-4.5,-2) {$D_8$};
\path[-stealth,font=\scriptsize]
(F.south) edge node[left] {$ inc $} (F'.north);
\end{tikzpicture} 
 \right) = 
\begin{tikzpicture}[descr/.style={fill=white},yscale=0.7,baseline=(O.base)] 
\node (O) at (-1,-1) {$ $};
\node (C) at (-2.5,0) {$\frac{C_2}{[C_2, D_8]}$};
\node (C') at (-2.5,-2) {$1$};
\path[-stealth,font=\scriptsize]
(C.south) edge node[left] {$  $} (C'.north);
\end{tikzpicture} 
 = 
\begin{tikzpicture}[descr/.style={fill=white},yscale=0.7,baseline=(O.base)] 
\node (O) at (-1,-1) {$ $};
\node (C) at (-2.5,0) {$\frac{C_2}{1}$};
\node (C') at (-2.5,-2) {$1$};
\path[-stealth,font=\scriptsize]
(C.south) edge node[left] {$  $} (C'.north);
\end{tikzpicture} 
 = 
\begin{tikzpicture}[descr/.style={fill=white},yscale=0.7,baseline=(O.base)] 
\node (O) at (-1,-1) {$ $};
\node (C) at (-2.5,0) {$C_2$};
\node (C') at (-2.5,-2) {$1$};
\path[-stealth,font=\scriptsize]
(C.south) edge node[left] {$  $} (C'.north);
\end{tikzpicture} 
  \] which is not the trivial crossed module.
  \end{proof}

\begin{remark}
Thanks to a general result for any semi-abelian category of \cite{GranScherer}, we know that if a nullification functor $\P_\A$ admits a fiberwise localization, then the kernels of their localization morphisms are $\sf A$-acyclic.
\end{remark}

\section*{Conclusion}

In this article, we studied fiberwise localization for regular-epi localization functors of crossed modules of groups. To sum up, we found 
an adequate normality condition on a short exact sequence and a regular-epi localization functor that guarantees the existence of a fiberwise localization. 
We proved that this condition can be expressed easily in $\sf XMod$. This simple statement allowed us to obtain examples of fiberwise localizations and
unexpectedly also counter-examples: the nullification functor $\P_{\X \mathbb{Z}}$ does not admit a fiberwise localization in general. 
This observation highlighted an interesting difference between the fiberwise localizations of crossed modules and groups, 
 which can be tracked down to the notion of characteristic
subobject, namely one which is not only normal, but invariant under all automorphisms.
If we have groups $K \leq N \leq T$ such that $K$ is a characteristic subgroup of $N$ and $N$ is a normal subgroup of $T$, then $K$ is a normal subgroup of $T$. However the corresponding statement does not hold for crossed modules, as we have seen in \cref{nonexistencelocfibr} (note that $\ker(\ell^\N)$ is a characteristic subcrossed module of $N$).

We observed 
another important difference between nullification functors of crossed modules and groups. In the category of groups, nullification functors 
$\sf P_A$ are characterized by the fact that the kernels of their localization morphisms are $\sf A$-acyclic. In $\sf XMod$, $ \P_{\X \mathbb{Z}}$ 
does not satisfy this property. 
It is interesting to notice that even if the functor of nullification $\P_{\X \mathbb{Z}}$ fails to have a fiberwise localization or 
to have $\X \mathbb{Z}$-acyclic kernels, its local objects form a Birkhoff subcategory. 

Fiberwise localization is an important tool in the study of conditional flatness: 
{\blue in \cite{GranScherer} and \cite{FS}, the authors used functorial fiberwise localization 
to study conditionally flat localizations.  Since we are not always able to construct a fiberwise 
localization for regular-epi localization functors of crossed modules, it is interesting to} 
investigate conditional flatness for such functors, which we do in the forthcoming paper \cite{MSS2}. 
We study how we can simplify this property by following the strategy of \cite{FS,GranScherer}. 
We prove that a regular-epi localization functor is conditionally flat if and only if it is 
admissible (in the sense of Galois) for the class of regular epimorphisms. We also prove 
that nullification functors are conditionally flat even if they fail to be characterized by 
an acyclicity condition.

\printbibliography
\end{document}